\documentclass[11pt,reqno]{amsart}
\usepackage{amsthm}
\usepackage{siunitx} 
\newcommand {\CC}  {{\mathbb C}}

\newcommand {\NN}  {{\mathbb N}}

\newtheorem {Th}{Theorem}

\newtheorem {Pro}{Proposition}
\newtheorem {Le}{Lemma}

\begin{document}
\title{Explicit bounds for composite lacunary polynomials}
\author{Christina Karolus}
\address[Christina Karolus]{University of Salzburg, Hellbrunnerstr. 34/I, 5020 Salzburg, Austria}
\email{christina.karolus@sbg.ac.at}

\begin{abstract}
Let $f, g, h\in \CC\left[x\right]$ be non-constant complex polynomials satisfying $f(x)=g(h(x))$ and let $f$ be lacunary in the sense that it has at most $l$ non-constant terms. Zannier proved in \cite{za2} that there exists a function $B_1(l)$ on $\NN$, depending only on $l$ and with the property that $h(x)$ can be written as the ratio of two polynomials having each at most $B_1(l)$ terms. Here, we give explicit estimates for this function or, more precicely, we prove that one may take for instance 
\begin{center}$B_1(l)=(4l)^{(2l)^{(3l)^{l+1}}}$.\end{center}
Moreover, in the case $l=2$, a better result is obtained using the same strategy.   
\end{abstract}
\maketitle

\section{Introduction}
\label{intro}

Let $f,g,h\in\CC\left[x\right]$ and $f=g\circ h$ be a lacunary polynomial with $l$ non-constant terms, i.e. $f$ is of the form 
$f(x)=a_0+a_1x^{n_1}+\cdots +a_lx^{n_l}$. Note that only the number of terms is viewed as fixed, while the coefficients and the degrees may vary. 
In \cite{za2}, it was shown by Zannier that there exists a function $B_1(l)$ such that $h(x)$ can be written as the ratio of two polynomials in $\CC\left[x\right]$ both having no more than $B_1(l)$ terms. In order to give explicit estimates for $B_1(l)$, we are following the stategy of \cite[Prop. 2]{za2}. Therefore, we will recall those parts of Zannier's proof, which are relevant to our arguments, in the very beginning of the paper. We prove

\begin{Th}\label{main}
Let $f, g, h\in\CC\left[x\right]$ be non-constant complex polynomials such that $f(x)=g(h(x))$ has at most $l$ non-constant terms. Then $h(x)$ can be written as the ratio of two polynomials in $\CC\left[x\right]$ having each at most $B_1(l)$ terms, where  
\begin{center}
$B_1(l)=(4l)^{(2l)^{(3l)^{l+1}}}$ for $l\ge 1$.
\end{center}
\end{Th}

There are quite different notions of {\it lacunarity} (lacunary polynomials are also sometimes called {\it sparse}).
Here, we deal with the situation that the number of terms of a given polynomial is fixed. It was conjectured by Erd\H{o}s that if $g$ is a complex non-constant polynomial with the property that $g(x)^2$ has at most $l$ terms, then $g(x)$ has also boundedly many terms and their number depends only on $l$. In \cite{sch}, Schinzel proved a generalized version of Erd\H{o}s's conjecture, namely the statement not only for $g(x)^2$ but for $g(x)^d$, $d\in\NN$. He also extended the conjecture to compositions $f(x)=g(h(x))$, claiming that if $f$ has $l$ terms, then $h(x)$ has at most $B(l)$ terms for some function $B$ on $\NN$. Zannier gave a proof for this (actually in a stronger version), wherein he showed in a first step the existence of a function $B_1$ such that, under the given assumptions, $h(x)$ can be written as a rational function with at most $B_1(l)$ terms in both the numerator and the denominator \cite{za2}. Using this result, he proved the stated claim for representations as polynomials\footnote{Examples as $h(x)=(x^n-1)/(x-1)=x^{n-1}+x^{n-2}+\cdots+x+1$ show that, written as a polynomial, $h(x)$ can have substantially more terms than in a representation as a rational function.} and moreover, he gave a complete description of general decompositions $f(x)=g(h(x))$ of a given polynomial $f(x)$ with $l$ terms. Also, for an outer composition factor $g$ in $f(x)=g(h(x))$, Zannier gave suitable bounds for the degree of $g$ depending only on the number of terms of $f$ \cite{za1}. Note that in both, the polynomial case and the case of a rational function, the special shapes $h(x)=ax^m+b$ and $h(x)=ax^m+bx^{-m}+c$, respectively,  must be taken into account. This follows easily from the following observation. Let $h(x)=ax^m+b$ and $g(x)=g_1(x-b)$. Then $f(x)=g(h(x))=g_1(ax^m)$ has at most $l$ non-constant terms, whereas the degree of $g$ can be arbitrary high. Similarly, in the case of rational functions, one can take for instance $h(x)=x+x^{-1}$ and $g(x)=T_n(x)$, the $n$-th Chebyshew-polynomial, to get a contradiction with the given statements. However, the results were later extended first to Laurent-polynomials \cite{za3} and then to rational functions \cite{fz}. 
Recently, in \cite{fmz} Fuchs, Mantova and Zannier achieved a final result for completely general algebraic equations $f(x,g(x))=0$. Here, $f(x,y)$ is assumed to be monic and of given degree in $y$ and with boundedly many terms in $x$. As pointed out in \cite{fp}, there are also other forms of lacunarity. Here, Fuchs and Peth\H{o} considered rational functions having only a given number of zeros and poles and they again studied their decomposability. This can be seen as a multiplicative analogue to the above mentioned problem. Based on their results, by computational experiments Peth\H{o} and Tengeley studied the decomposability of rational functions having at most four zeros and poles and they provided parametrizations of all possible solutions and the appropriate varieties in this case \cite{pt}.

The present paper is organized as follows. In the very beginning, we give the main results and parts of Zannier's proof, which are crucial for our deductions. Based on this, in Section 3 we prove Theorem \ref{main}, giving an explicit bound for the function $B_1(l)$ in Zannier's Proposition. This bound happens to be triple-exponential. The reason for this is the following. In the proof, an estimate for the ratio $n_l/n_1$ is used to give an upper bound for an exponent. Such an estimate can be found through an recursive procedure, pointed out by Zannier. The bound we obtain through this method will be double-exponential, which in the end leads to the received order. We do not know whether that can be improved in general or how far away we are from a ``good bound''. However, our bound surely is far from the truth.  
At least for the cases $l=1$ and $l=2$ we have the smaller estimates $2$ and $1\ 114\ 112$, respectively. For the latter bound, the corresponding statement and its proof is given in the last part of Section 3.

Finally, we mention that, independently of us, Dona in \cite{do} also gave an explicit bound for $B_1(l)$. His deductions are based on Zannier's proof too. Therefore it is not surprising that qualitatively his bound is of the same shape (i.e. also triple-exponential) although it is quantitatively better than ours. Nevertheless, and also for the fact that our result is already mentioned in \cite{fh}, it appears to us that the result is still worth to be found in the literature.

\section{Arguments from Zannier's proof}
\label{sec:1}

In this section we recall the main steps of Zannier's proof \cite[Prop. 2]{za2}. For clarity, we keep the original notation.  
Assuming that $f(x)=g(h(x))$ has at most $l$ non-constant terms, let $\deg f=m=n_l$ and $\deg g=d$, so that we have $\deg h=n_l/d$. From \cite[Thm. 1]{za1} it follows that $d\le 2l(l-1)$. Set $y=1/x$ and $\tilde{h}(y)=x^{-n_l/d}h(x)=y^{n_l/d}h(1/y)$, such that $\tilde{h}\in\CC\left[y\right]$. Moreover, write $f(x)=ax^m(1+b_1y^{n_1}+\cdots +b_ly^{n_l})$, where $0=:n_0<n_1<\ldots<n_l$. We may assume that $f$ has exactly $l$ non-constant terms, i.e. $ab_1\cdots b_l\neq0$. Instead of $h(x)$, Zannier proves the statement for $\tilde{h}(y)$, which in fact has the same number of terms as $h(x)$. Also, its degree is bounded by $\deg \tilde{h}\le n_l/d=\deg h$.

Writing $\delta_p(x)=1+b_1x^{n_1}+\cdots+b_px^{n_p}$ for an integer $p$ with $1\le p \le l-1$,
Zannier deduces that, in $\CC\left[\left[y\right]\right]$, $\tilde{h}$ is of the form $\tilde{h}(y)=t_1+t_2+\cdots +t_L+O(y^{2n_l})$, where $L$ is an integer which may be bounded by a function in $l$ and $n_l/n_{p+1}$ and the $t_1,\ldots,t_L$ are all of the shape
\begin{align}
c\delta_p(y)^{s/d-k} y^{h_1n_{p+1}+\cdots +h_{l-p}n_p+(1-s)m/d},\label{eqn:delta}
\end{align}
for varying $h_1,\ldots,h_{l-p}\in\NN$, $k=h_1+\ldots+h_{l-p}$, suitable constants $c=c(h_1,\ldots,h_{l-p},
s)$ and $s\in \{1,0,-1,\ldots,1-2d\}$, where the exponent of $y$ in such terms is smaller than $2n_l$. An easy argument shows that for $L$ one may take the rough estimate  $L\le(2d+1)(2n_l/n_{p+1}+1)^l$. 

In the case that $t_1,\ldots,t_L,\tilde{h}(y)$ are linearly independent over $\CC$, the author then proves that
\begin{align}
n_l\le 16^{l+1}d^3(n_l/n_{p+1})^{2l}(1+n_p),\label{quotient2}
\end{align}
and furthermore, using $1+n_p\le 2 n_p$, it follows that
\begin{align}
(n_l/n_p)\le2\cdot16^{l+1}d^3(n_l/n_{p+1})^{2l}.\label{quotient1}
\end{align}

If on the other hand $t_1,\ldots,t_L,\tilde{h}(y)$ are linearly dependent over $\CC$, it inductively follows that $\tilde{h}(y)$ can be written as the ratio of two polynomials in $\CC\left[y\right]$ having each at most $B_2(l,n_l/n_{p+1})$ terms, where $B_2(l,u)$ is a suitable function which may be estimated in terms of $B_1(l-1)$ and of $u\ge0$. 

Zannier then distinguishes between those two cases for each $p=l-1,l-2,\ldots$. Suppose that for $p=l-1,l-2,\ldots,l-r$ always the first situation occurs. In that case we can recursively determine upper bounds for the quotients $n_l/n_p$, since for $p=l-1$ the initial condition $n_l/n_{p+1}=1$ holds. 

The proof then argues via backwards induction on $p=l-1,l-2,\ldots$. If $t_1,t_2,\ldots,t_L,\tilde{h}(y)$ are linearly independent {\it for all} $p=l-1,l-2,\ldots,1$, one can use (\ref{quotient1}) and (\ref{quotient2}) to get an estimate for $n_l$, which in fact also gives an estimate for the number of terms of $\tilde{h}(y)$ (and hence of $h(x)$) even written as a polynomial, as the number of terms of $\tilde{h}$ is bounded by its degree $n_l/d\le n_l$. 

On the other hand, if the $t_1,\ldots,t_L,\tilde{h}(y)$ are linearly dependent over $\CC$ for at least one $p\in\{1,\ldots,l-1\}$, let $p_0$ denote the last $p$ for which this case occurs, i.e. we may assume that we have linear independency for $p=l-1,l-2,\ldots,p_0+1$ and linear dependency for $p=p_0$. Here, Zannier concludes that $\tilde{h}(y)$ is the sum of at most $L=L(p_0)$ terms of the form (\ref{eqn:delta}), where $k\le 2l\cdot n_l/n_{p_0+1}$ and $s\in\{1,0,-1,\ldots,1-2d\}$.
Now the author uses the linear independency in the cases $p>p_0$ to estimate the quotient $n_l/n_{p_0+1}$, which occurs in the estimates for $L$ and $k$ to get the disired result.

\section{Proof of Theorem \ref{main} and the case $l=2$}
\label{sec:3}

In order to prove Proposition 1, we start with the following lemma.

\begin{Le}\label{lem}
Let $f\in\CC\left[x\right]$ be of the form $f(x)=a_0+a_1x^{n_1}+\cdots+a_lx^{n_l}$, $0<n_1<\ldots<n_l$, and let for every integer $p$, $1\le p\le l-1$, $S=S(p)=\{t_1,\ldots,t_L,\tilde{h}(y)\}$ be the set described in Section 2. If for each $p=l-1,\ldots,l-r$ the set $S$ is linearly independent over $\CC$, it holds that
\begin{center}
$\displaystyle\frac{n_l}{n_{l-r}}\le(16^{l+2}l^6)^{(3l)^{r-1}}$.
\end{center}
\end{Le}

\begin{proof}
We set $\lambda=2\cdot16^{l+1}d^3$. By \cite[Thm. 1]{za1}, we have $d\le2l(l-1)$. Applying (\ref{quotient1}), we obtain iteratively
\begin{center}
\renewcommand{\arraystretch}{2}
\begin{tabular}{l}
$\displaystyle\frac{n_l}{n_{l-1}}\le2\cdot16^{l+1}\cdot d^3=:\lambda$,\\
$\displaystyle\frac{n_l}{n_{l-2}}\le2\cdot16^{l+1}\cdot d^3(2\cdot16^{l+1}\cdot d^3)^{2l}=\lambda^{2l+1}$,\\
$\displaystyle\frac{n_l}{n_{l-3}}\le \lambda(\lambda^{2l+1})^{2l}=\lambda^{1+2l(1+2l)}$,\\
$\phantom{ba}\vdots$\\
$\displaystyle\frac{n_l}{n_{l-r}}\le\lambda^{1+2l(1+2l(1+2l(\ldots)))}\le\lambda^{(3l)^{r-1}}$.\\
\end{tabular}
\end{center}

As $\lambda=2\cdot16^{l+1} d^3\le2\cdot16^{l+1}(2l(l-1))^3\le16^{l+2}l^6$, in the case that for $p=l-1,l-2,\ldots,l-r$ the set $S$ is always linearly independent over $\CC$, we obtain
the claimed result.\end{proof}

\begin{proof}[Proof of Proposition 1]

Following the proof of \cite[Prop. 2]{za2}, we argue by induction on $l$. As in the previous section, we keep the notation from Zannier's proof. As pointed out by Zannier, for $l=1$ we may take $B_1(1)\ge 2$. Now, for the rest of the proof let us assume that the statement holds for $l-1$, i.e. that $B_1$ has been suitably defined on $\{1,2,\ldots,l-1\}$.

For $p=l-1,l-2,\ldots,0$ and $\delta_p(y)=1+b_1y^{n_1}+\cdots+b_py^{n_p}$, in $\CC\left[\left[y\right]\right]$ we can write 
\begin{align*}
\tilde{h}(y)=t_1+t_2+\cdots+t_L+O(y^{2n_l}),
\end{align*}
where the $t_i$, $1\le i\le L$, are terms of the shape (\ref{eqn:delta}), which we may assume to be linearly independent over $\CC$, and $L$ is an integer which, as we know from the proof of \cite[Prop. 2]{za2}, can be bounded by $L\le (2d+1)(2 n_l/n_{p+1}+1)^l$. 

We now consider the set $S=\{t_1,\ldots,t_L,\tilde{h}(y)\}$ and, following the original proof, we distinguish between two possible cases, namely that $S=S(p)$ is linearly independent over $\CC$ for every $p=l-1,l-2,\ldots,1$ or that there is an integer $p_0$, $1\le p_0\le l-1$, such that we have linear dependency, i.e. $S=S(p)$ is linearly independent for $p=l-1,l-2,\ldots,p_0+1$ and $S(p_0)$ is linearly dependent over $\CC$. 

{\it Case 1}. In the case of linear independency for every $p=l-1,l-2,\ldots,1$, by Lemma \ref{lem}, we get the estimate $n_l/n_1\le(16^{l+2}l^6)^{(3l)^{l-2}}$. 
Now, we may apply (\ref{quotient2}) for $p=0$ (recall that $n_0=0$) to get
\begin{align}
n_l&\le 16^{l+1}d^3(16^{l+2}l^6)^{(2l)(3l)^{l-2}}\notag\\
&\le 16^{l+1}(2l^2)^3(16^{l+2}l^6)^{(2l)(3l)^{l-2}}\notag\\
&<(16^{l+2}l^6)^{(2l)(3l)^{l-2}+1}\notag\\
&<(16^{l+2}l^6)^{(3l)^{l-1}}\notag.
\end{align}
This clearly gives an estimate for the number of terms of $\tilde{h}(y)$ written as a polynomial (and hence of $h(x)$), bounding its degree by $n_l/d\le n_l$. 

{\it Case 2}. Let us now consider the second case, where we have linear dependency for some $p=p_0$. Since we assume $S=S(p)$ to be linear independent over $\CC$ for each $p>p_0$, we have, again by Lemma \ref{lem}, that $n_l/n_{p_0+1}$ can be bounded by $(16^{l+2}l^6)^{(3l)^{l-1}}$ as well. Let $e=\left[K:\CC(y)\right]$, where $K=\CC(y,\delta_p(y)^{1/d})$, so $e\in\{1,\ldots,d\}$ is the least integer such that $\delta_p(y)^e$ is a $d$-th power in $\CC(y)$. We write $\delta_p(y)^e=\eta_p(y)^d$ for a polynomial $\eta_p\in\CC\left[y\right]$ to express this fact. Also, for the rest of the proof, we simply write $p$ instead of $p_0$. From Zannier's proof it follows that $\tilde{h}(y)=\Lambda_0$, where $\Lambda_0$ is the sum of at most $L$ terms of the shape 
\begin{align}
c\eta_p(y)^{(s-kd)/e}y^{(1-s)m/d+h_1n_{p+1}+\cdots+h_{l-p}n_l},\label{eqn:terms}
\end{align}
with $k=h_1+\ldots+h_{l-p}\le2l n_l/n_{p+1}$ and where $s\in\{1-2d,\ldots,0,1\}$ is such that $e|s$.  

In order to estimate the number of terms in the requested representation of $\tilde{h}$, we look at 
$\eta_p(y)^{(s-kd)/e}$, which is the important quantity in (\ref{eqn:terms}), when it comes to counting terms. 
Note that $\eta_p(y)^{d/e}=\delta_p(y)$ is a polynomial which has at most $p\le l-1$ non-constant terms, so by the induction hypothesis $\eta_p(y)$ can be expressed as a rational function with at most $B_1(l-1)$ terms in both the numerator and the denominator. That is,  
$$\eta_p(y)=\frac{\eta_{p,1}(y)}{\eta_{p,2}(y)},$$
where $\eta_{p,1},\eta_{p,2}$ are complex polynomials with at most $B_1(l-1)$ terms. 

To find a suitable function for $B_1$, we start with estimating the exponent of
$\eta_p(y)$ in (\ref{eqn:terms}). Recall that $s\in\{1-2d,\ldots,-1,0,1\}$, $k=h_1+\ldots+h_{l-p}<2ln_l/n_{p+1}$ and that $d\le 2l(l-1)$. Therefore we get the following
\begin{align*}
|(s-kd)/e|&\le|s|+|kd|\le 2d-1+2l\frac{n_l}{n_{p+1}}d\\
&\le 2\cdot2l(l-1)\left(1+l\frac{n_l}{n_{p+1}}\right)-1\\
&\le 2^2l^2\left(1+l(2^{4(l+2)}l^6)^{(3l)^{l-1}}\right)-1\\
&<2^2l^2\left(2\cdot2^{(4l+8)(3l)^{l-1}}l^{(3l)^l}\right)-1\\
&=2^{3+(4l+8)(3l)^{l-1}}l^{2+(3l)^l}-1.
\end{align*}

Hence, if we set $M=2^{3+(4l+8)(3l)^{l-1}}l^{2+(3l)^l}$, we see that 
\[(s-kd)/e\in\{-(M-1),\ldots,-1,0,1\}.\]

Now, let us again consider $\tilde{h}(y)=\Lambda_0$.
After reducing all of the $L$ terms of the shape (\ref{eqn:terms}) to the common denominator $\eta_{2,p}(y)\eta_{1,p}(y)^{M-1}$, we can make the rough estimate $B_1(l-1)B_1(l-1)^{M-1}=B_1(l-1)^M$ for the number of terms in the denominator and $LB_1(l-1)^M$ for the number of terms in the numerator.
Since we are looking for a function that bounds both the number of terms in the numerator and the denominator, it now suffices to define $B_1$ in such a way that $B_1(l)\ge x_l$, where 
$(x_l)_{1\le l}\subset\NN$ is the recurrence sequence defined by 
$x_1=2$ and $x_l=L\cdot x_{l-1}^M$ (recall that for $l=1$ it already has been shown in \cite[Prop. 2]{za2} that we may take $B_1(1)\ge 2$).
It follows that 
\[x_l=L^{1+M+M^2+\ldots +M^{l-2}}\cdot 2^{M^{l-1}}<(2L)^{M^{l-1}},\]
hence we may as well define $B_1$ as any function satisfying 
$B_1(l)\ge(2L)^{M^{l-1}}$.
Similarly as for the exponent, we find an upper bound for $2L$,
\begin{align*}
2L&\le 2(2d+1)(1+2n_l/n_{p+1})^l\\
&\le2(2^2l(l-1)+1)(1+2(2^{4(l+2)}l^6)^{(3l)^{l-1}})^l\\
&<2^3l^2\cdot2^{2l+4l(l+2)(3l)^{l-1}}l^{2(3l)^l}\\
&=2^{3+2l+4l(l+2)(3l)^{l-1}}l^{2(3l)^l+2}\\
&=4^{1.5+l+2l(l+2)(3l)^{l-1}}l^{2(3l)^l+2}\\
&<4^{(3l)^{l-1}(2l^2+4l+1)}l^{2((3l)^l+1)}\\
&<4^{(3l)^{l-1}(4/3)^l3l^2}l^{2(4l)^l}\\
&=4^{2^{2l}l^{l+1}}l^{2^{2l+1}l^l}.
\end{align*}
Based on these estimates, we get an upper bound for $(2L)^{M^{l-1}}$ and therefore also for $B_1(l)$:
\begin{align*}
(2L)^{M^{l-1}}&\le(4^{2^{2l}l^{l+1}}l^{2^{2l+1}l^l})^{2^{(3+(4l+8)(3l)^{l-1})(l-1)}l^{(2+(3l)^l)(l-1)}}\\
&=4^{2^{(3l)^{l-1}(4l^2+4l-8)+5l-3}l^{(3l)^l(l-1)+3l-1}}l^{2^{(3l)^{l-1}(4l^2+4l-8)+5l-2}l^{(3l)^l(l-1)+3l-2}}
\end{align*}
Note that for the exponents we have
\begin{align*}
(3l)^{l-1}(4l^2+4l-8+(5l-2)/(3l)^{l-1})<(3l)^{l-1}(3l)^2=(3l)^{l+1}
\end{align*}
and
\begin{align*}
(3l)^l(l-1)+3l-1&<(3l)^{l}(l-1+1)<(3l)^{l+1}.
\end{align*}
We therefore get
\begin{align*}
(2L)^{M^{l-1}}&\le 4^{2^{(3l)^{l+1}}l^{(3l)^{l+1}}}l^ {2^{(3l)^{l+1}}l^{(3l)^{l+1}}}=(4l)^{(2l)^{(3l)^{l+1}}}.
\end{align*}
Finally, we have to check that the obtained estimate also holds in {\it Case 1}, i.e. if $(16^{l+2}l^6)^{(3l)^{l-1}}\le (4l)^{(2l)^{(3l)^{l+1}}}$ for $l\ge 2$. 
But this follows trivially from the fact, that we already used the quantity $(16^{l+2}l^6)^{(3l)^{l-1}}$ in estimating $2L$ in {\it Case 2}, hence the claimed result.
\end{proof}

Using the same arguments, it is also possible to obtain a better bound for $l=2$, since in this case we are able to keep the estimates during the proof essentially smaller. 
Eventually, we get the following.

\begin{Pro}
Let $f, g, h\in\CC\left[x\right]$ be non-constant complex polynomials such that $f(x)=g(h(x))$ has at most $2$ non-constant terms. Then $h(x)$ may be written as the ratio of two polynomials in $\CC\left[x\right]$ having each at most $1\ 114\ 112$ terms.
\end{Pro}

\begin{proof}
Assume that $f$ is a polynomial in $\CC\left[x\right]$ with at most two non-constant terms, i.e. $f(x)=a_0+a_1x^{n_1}+a_2x^{n_2}$, $0<n_1<n_2$. We keep the notations from above.
If $\deg g=d=1$, the number of terms of $g(h(x))$ and of $h(x)$ may only diviate by one (namely the constant term), hence in this case $h(x)$ is a polynomial that has not more than 3 terms. 
So in the following we assume that $d\ge 2$. Note that we also have $d\le 4$, for $d\le2l(l-1)$.

As before, we start with the observation that $\tilde{h}(y)$ can be written in the shape 
\[\tilde{h}(y)=t_1+t_2+\ldots+t_L+\mathcal{O}(2n_2),\]
where the $t_i$ are terms of the form \eqref{eqn:delta}, which we may assume to be linearly independent over $\CC$. 
We again consider the two cases that $S=S(p)=\{t_1,\ldots,t_L,\tilde{h}(y)\}$ is linearly dependent or linearly independent over $\CC$, respectively. In general, we had to distinguish between those cases for each $p=l-1,l-2,\ldots$, but since $l=2$, $p=1$ is the only remaining situation we have to look at. The above approximation was chosen in such a way that in each $t_i$ the exponent of $y$ is smaller than $2n_2$, that is 
\[\frac{(1-s)n_2}{d}+h_1n_2=\left(\frac{1-s}{d}+h_1\right)n_2\le 2n_2,\]
where $s\le1$ and $h_1\ge0$ are integers. It follows that $1-2d\le s\le 1$ and $h_1\in\{0,1,2\}$. A rough estimate on the number of such terms $t_i$ would therefore be $3(2d+1)\le27$,
but checking for each $d\in\{2,3,4\}$ seperately, we find that under the given conditions there cannot be more than 17 such exponents, hence
\[L\le17.\]

{\it Case 1.} If $\{t_1,\ldots,t_L,\tilde{h}(y)\}$ is linearly independent over $\CC$, then  by \cite[Prop.1]{za2} we get
\[n_2\le L(L+1)d(1+n_1)\le17\cdot18\cdot4\cdot2n_1,\]
and consequently $n_2/n_1\le2448$. 
The starting point in Zannier's proof was to expand $h(x)$ as a Puiseux-series, which led to the observation that, in $\CC\left[\left[y\right]\right]$, for certain $\gamma_{-1}, \gamma_0,\gamma_1,\ldots\in\CC$ we have
$$\tilde{h}(y)=\gamma_{-1}\tilde{f}(y)^{1/d}+\gamma_0 y^{m/d}+\gamma_1y^{2m/d}\tilde{f}(y)^{-1/d}+\gamma_2y^{3m/d}\tilde{f}(y)^{-2/d}+\ldots,$$
where the roots $\tilde{f}(y)^{s/d}$ are of the form
$$\tilde{f}(y)^{s/d}=\sum\limits_{(h_1,h_2)\in\NN_0^2}c_{s,d,(h_1,h_2)}b_1^{h_1}b_2^{h_2}y^{h_1n_1+h_2n_2},$$
for certain universal coefficients $c_{s,d,(h_1,h_2)}$. Since $\tilde{h}(y)$ is a polynomial with $\deg \tilde{h}\le m/d=n_2/d$, it follows that the only terms which may contribute to $\tilde{h}(y)$ are $\gamma_0y^{m/d}$ and terms coming from $\gamma_{-1}\tilde{f}(y)^{1/d}$, for which $h_1n_1+h_2n_2\le n_2/d$ holds. Since we assumed $2\le d$, it follows that $h_2=0$ and $h_1\le n_2/(2n_1)=1224$. Taking also the terms with $h_1=h_2=0$ and $\gamma_0y^{m/d}$ into account, we conclude that $\tilde{h}$ is a polynomial having no more than 1226 Terms.\\

{\it Case 2.} On the other hand, if $\{t_1,\ldots,t_L,\tilde{h}(y)\}$ is linearly dependent over $\CC$, then $\tilde{h}(y)$ may be written as the sum of at most $L\le17$ terms of the shape \eqref{eqn:terms}, where $\eta_p(y)=\frac{\eta_{p,1}(y)}{\eta_{p,2}(y)}$ and $\eta_{p,1}(y),\eta_{p,2}(y)\in\CC\left[x\right]$ both have at most two terms and $k=h_1\le2$. Recall that $s\in\{1,0,\ldots,-7\}$ and $d\le4$. It follows that $(s-kd)/e\in\{1,0,-1,\ldots,-15\}$. Bringing all of the terms to common denominator $\eta_{p,1}(y)^{15}\eta_{p,2}(y)$, we can write $\tilde{h}(y)$ as a rational function with at most $2^{16}$ terms in the denominator and $L\cdot2^{16}\le17\cdot 2^{16}=1\ 114\ 112$ terms in the numerator.  
\end{proof}

The obtained bound is still not very small. However, compared to the bound $B_1(2)=2^{3\cdot 2^{432}}$, where $2^{3\cdot 2^{432}}>10^{2^{431}}$ and the exponent $2^{431}$ already has 130 digits, this still gives a noteable improvement. \\

\section*{Acknowledgement} The author was supported by Austrian Science Fund (FWF) Grant No. P24574.

\end{document}